\documentclass[11pt, a4paper]{amsart}
\usepackage{verbatim}
\usepackage[dvips]{color}
\usepackage{amssymb}
\usepackage[active]{srcltx}
\usepackage{mathrsfs}
\usepackage{latexsym}
\usepackage{amssymb}
\usepackage{graphicx}
\usepackage{geometry}
\usepackage{amsmath}
\usepackage{float}
\usepackage{mathrsfs}
\usepackage[active]{srcltx}
\usepackage{xcolor}

\usepackage{epstopdf}
\usepackage{caption}

\allowdisplaybreaks

\newtheorem{theorem}{Theorem}[section]
\newtheorem{proposition}{Proposition}[section]
\newtheorem{lemma}{Lemma}[section]

\newtheorem{assumption}{Assumption}[section]
\newtheorem{definition}{Definition}[section]
\newtheorem{remark}{Remark}[section]

\theoremstyle{definition}

\floatstyle{plain} \restylefloat{figure} \restylefloat{table}

\def\A{\mathcal A}

\def\C{\mathcal C}

\def\E{\mathbb E}
\def\F{\mathcal F}

\def\I{\mathcal I}
\def\J{\mathcal J}

\def\P{\mathbb P}

\def\R{\mathbb R}
\def\S{\mathcal S}
\def\T{\mathcal T}

\def\ud{\mathrm{d}}
\def\OC{\overline{\mathcal{C}}}

\title{Playing with ghosts in a Dynkin game}
\author[Tiziano De Angelis and Erik Ekstr\"om]{Tiziano De Angelis and Erik Ekstr\"om}
\address{T.~De Angelis: School of Mathematics, University of Leeds, LS2 9JT Leeds, United Kingdom.}\email{t.deangelis@leeds.ac.uk}
\address{E.~Ekstr\"om: Department of Mathematics, Uppsala University, Box 480, 75106 Uppsala, Sweden.}
\email{ekstrom@math.uu.se}
\date{\today}
\thanks{{\em Acknowledgments}: T.~De Angelis gratefully acknowledges support by the EPSRC grant EP/R021201/1 and  E.~Ekstr\"om by the Knut and Alice Wallenberg Foundation.
Parts of this work were carried out during T.~De Angelis' visits to Uppsala University and E.~Ekstr\"om's visits to University of Leeds. We thank both institutions for their hospitality.}

\begin{document}

\begin{abstract}
We study a class of optimal stopping games (Dynkin games) of preemption type, with uncertainty about the existence of competitors.
The set-up is well-suited to model, for example, real options in the context of investors who do not want to publicly reveal their interest in a certain business opportunity. 
We show that there exists a Nash equilibrium in randomized stopping times which is described explicitly in terms of the corresponding one-player game. 
\end{abstract}

\maketitle

\section{Introduction}

\subsection*{Background} Stopping games have received huge attention in the stochastic control literature since their inception, dating back to work by Dynkin in \cite{D}. In the standard modern formulation of the two-player game (due to Neveau \cite{Ne}) the players have gains/losses depending on a stochastic process $X$, which they both observe. Their aim is to maximise gains (or minimise losses) by finding stopping rules that allow for a Nash equilibrium in the game. Both players know the structure of the game and have full information on the specifications of the process $X$. 

Many real-world applications call for incomplete and/or asymmetric information about the game structure and/or the underlying stochastic process. In particular, in this paper we are interested in determining equilibria for two-player Dynkin games in which each player is uncertain about the existence of a competitor. Before describing our contribution in further detail we spend a few words on the existing literature in order to contextualise the problem.

It is difficult to provide a detailed literature review that would do justice to the numerous contributions on Dynkin games in the standard framework (i.e., full and symmetric information) and it falls outside the scope of our introduction. For the case of zero-sum games one may for example refer to \cite{LM} and \cite{EP} for general treatments of the martingale and the Markovian set-ups, respectively, and to \cite{K} for a reduction of a financial game option into a stopping game. Other seminal contributions to this literature can be found in \cite{B}, \cite{S} and \cite{Ya} among others. For the case of nonzero-sum games one may refer to \cite{BF} for connections between such games and variational inequalities, to \cite{HZ} for martingale methods for existence of a saddle-point in a general set-up, and to \cite{At} and \cite{DFM} for sufficient conditions for the existence (and uniqueness) of Nash equilibria in hitting times to thresholds for one dimensional diffusions.

In our paper we will be dealing with a class of problems that share similarities with nonzero-sum Dynkin games. However we depart from the standard set-up by including the key feature of {\em uncertainty about competition}. In this respect we draw from the literature on games with incomplete/asymmetric information, whose main common denominator is the need of the players to hide their information from the competitors. Mathematically this translates into the use of randomised stopping times; the latter can be informally understood as stopping rules which prescribe to stop according to some `intensity'; for example, in a discrete-time setting, it means that stopping may occur at each time with some probability. The other key feature of this type of games is the need to account for the dynamic evolution of the players' beliefs concerning those parameters that they cannot fully observe (in other words, the players update their private views on the `state of the world' by observing what happens during the game, and one needs to keep track of such updates). We incorporate beliefs in our game by adding a state variable, denoted by $\Pi$, whose dynamic is constructed starting from the randomised stopping strategies of the two players (see Section \ref{sec:construction} below).

The literature on Dynkin games with asymmetric/incomplete information has started to gain traction in recent years. The first contribution that we are aware of is by Gr\"un \cite{G}. In \cite{G}, a zero-sum stopping game with asymmetric information about the payoff functions is considered. In a setting where one of the players has the informational advantage of knowing the payoff functions and the other player only knows the distribution of possible payoff functions, a value of the game is obtained (with randomised stopping times) and characterised as a viscosity solution of a nonlinear variational problem. Later on, in \cite{GG}, authors use methods inspired by dynamic programming to study a zero-sum stopping game in which the players have access to different filtrations generated by the dynamics of two different processes. Explicitly solvable examples in this literature are still rare and the first one is given in \cite{DEG} (we provide another one in Section \ref{sec:realoption} for our game). Authors in \cite{DEG} study a zero-sum stopping game with asymmetric information about a drift parameter of the underlying diffusion process. An explicit Nash equilibrium is obtained in which the uninformed player uses a normal stopping time (pure strategy) and the informed one uses a randomised stopping time (mixed strategy). 

For completeness, but without elaborating further, we also mention that there exists a vast literature on stochastic differential games with asymmetric information and the interested reader may look into \cite{C} and \cite{CR} and references therein.

In a strategic context of agents playing hide and seek, it seems natural to ask what happens if players cannot be certain about the existence of competition. Several real-world situations fall under this category as, for example, (a) investors who do not want to reveal their interest for specific business opportunities (so-called real options), (b) potential house buyers who are not aware of how many other offers will be put forward (and how quickly), (c) buyers of depletable assets/goods (e.g., cheap flight tickets), etc. The feature of uncertain competition has indeed been addressed in a static setting of auction theory, see, e.g.~\cite{HKL} and \cite{MM}, and more recently \cite{EL}. However, to the best of our knowledge there are no studies featuring uncertain competition in a dynamic setting and in particular there seem to be no contributions to the theory of Dynkin games. With this paper we aim at filling that gap and encourage further research in that direction.

\subsection*{Our contribution}
Here, we study a combined stopping/preemption game between two players who are interested in the same asset. The player who stops first receives the full payoff, defined in terms of an underlying Markov process $X$ representing the asset. The game is a {\em nonzero-sum} Dynkin game but, in contrast with the classical set-up, in our model the players face {\em uncertain competition}, i.e. each player is uncertain as to whether the other player exists or not. 

At the start of the game each player estimates the probability of competition. That is, Player 1 believes she has competition with probability $p_1$ and Player 2 believes she has competition with probability $p_2$. As the game evolves, both players adjust their beliefs according to the dynamic of their own belief process $\Pi^i$, with $i=1,2$. Such adjustment is based on a combination of two key elements: (i) the observation of the underlying asset $X$ and (ii) the lack of action from the other player. Intuitively, if the payoff associated with the current asset value becomes large, this is appealing for both players; therefore, from the point of view of Player 1, the fact that Player 2 has not stopped yet, suggests that Player 2 may not exist at all. (This simple heuristics also motivates the title of our paper.) 

Within this context the use of randomised stopping times stems from two observations. On the one hand, it allows the players to hide their participation in the game in order to `fool' their opponent. On the other hand, it is intuitively clear that, due to the preemption feature of the game, it would be impossible to reach a non-trivial equilibrium using pure stopping times (with respect to the filtration generated by the asset). Indeed, if Player 1 picks a stopping time $\tau$, then Player 2 would possibly stop just before $\tau$ so as to receive the full payoff.

In this paper, with no loss of generality we consider $p_1\le p_2$. Then we prove that there exists a Nash equilibrium in terms of strategies whose character completely depends on the initial belief of Player 1. Here we only describe the main ideas around the structure of the equilibrium but we emphasise that, at a deeper level, we find several remarkable properties of the players' optimal strategies which will be described in fuller detail in Section \ref{sec:comment} (as they need a more extensive mathematical discussion). 

It turns out that the state space of the two-dimensional process $(\Pi^1,X)$ divides into three disjoint regions: the no-action region ($\C$), 
the action region ($\C'$) and the stopping region ($\S$). If $(\Pi^1,X)$ starts in the no-action region (note that $\Pi^1_0=p_1$), then the equilibrium consists of no action (from either player) until the boundary $\partial\C$ of the region is reached (notice that no action results in $\Pi^i$ being constant for $i=1,2$); then Player 2 employs a randomised strategy in such a way that the process $(\Pi^1,X)$ reflects along the boundary $\partial\C$ towards the interior of $\C$ (as an effect of decreasing $\Pi^1$); at the same time Player 1 will follow a (equilibrium) strategy consisting of stopping with a fraction $p_1/p_2$ of the `stopping intensity' of Player~2.
If, instead, the initial beliefs are such that $(\Pi^1,X)$ starts in $\C'$, then Player 2 employs a strategy consisting of 
stopping immediately with a certain probability. That makes the process $(\Pi^1,X)$ jump strictly into the interior of the no-action region $\C$. After this initial jump the reflection strategy described above is employed. Also in this case Player 1 will follow the (equilibrium) strategy of stopping with a fraction $p_1/p_2$ of the `stopping intensity' of Player~2. Finally, if $(\Pi^1,X)$ starts in the stopping region $\S$, then the equilibrium strategy consists of immediate stopping for both players. In this case the players split the payoff evenly.
 
Remarkably, the boundaries of the no-action region and of the action-region can be specified explicitly in terms of the corresponding one-player game with no competition. This allows to construct explicitly randomised stopping times, at equilibrium, in all those games whose one-player counterpart is explicitly solvable. (For instance the literature on optimal stopping of one-dimensional diffusions is rich with such examples, see e.g.~\cite{PS}.) In particular, in Section \ref{sec:realoption} we provide the explicit solution of a real option game with uncertain competition as an illustration of our results. 

As explained above, our work is the first one to address uncertain competition in a dynamic setting. We find a surprisingly explicit yet rich structure of the equilibrium strategies, which allows for a collection of interesting considerations (see details in Section~\ref{sec:comment}). In short, we observe that the most active player is the one who has the largest initial belief about the existence of the opponent, and that she is the one that benefits the most from randomisation; we call this feature the `benefit of wariness'. 
Also, the strategy adopted by Player 2 when the game starts with $(\Pi^1_0,X_0)\in \C'$ is somewhat surprising: in the stopping literature there is no analogue for the region $\C'$ and in the singular control literature (which is also related to the present work) one should not expect jumps strictly in the interior of the no-action region $\C$ (at least in absence of fixed costs of control). Finally, it is possible to draw a parallel between the strategies that we construct and a concept of entry time to `randomised' sets. This is also rather surprising because, from the beginning, we look for Nash equilibria without any restrictions on the randomised stopping times.

The rest of the paper is organised as follows. The game is set in Section \ref{sec:setup}, where we also define randomised stopping times and recall the concept of Nash equilibrium in our context. In Section \ref{sec:obs} we derive a number of properties of the two players' expected payoffs, which are needed for the subsequent analysis. Section \ref{sec:construction} is devoted to the construction of the belief processes, the specification of the sets $\OC$ (which corresponds to $\C$ plus a portion of its boundary), $\C'$ and $\S$, and the construction of a suitably reflected belief process. The existence and the main additional facts around the Nash equilibrium are derived in Section \ref{sec:NE}. We conclude the paper with a fully solved example in Section \ref{sec:realoption}.

\section{Set-up}\label{sec:setup}

Let $(\Omega,\P,\F)$ be a probability space hosting the following: 
\begin{itemize}
\item[(a)] a continuous, $\R^d$-valued, strong Markov process $X$ which is regular (it can reach any open set in finite time
with positive probability, for any value of the initial point $X_0=x$); 
\item[(b)] two Bernoulli distributed random variables 
$\theta_i$, $i=1,2$; 
\item[(c)] two Uniform$(0,1)$-distributed random variables $U_i$, $i=1,2$.
\end{itemize}
Furthermore, we assume that these 
processes and random variables are mutually independent, and that $\P(\theta_i=1)=1-\P(\theta_i=0)=p_i\in(0,1]$.

We denote by $\F^X=\{\F^X_t\}_{0\leq t<\infty}$ the right-continuous augmentation of the filtration generated by $X$, and
we let $\T$ be the set of $\F^X$-stopping times. In what follows we will use the notations $\P_x(\,\cdot\,):=\P(\,\cdot\,|X_0=x)$ and $\E_x[\,\cdot\,]:=\E[\,\cdot\,|X_0=x]$ 
for $x\in\R^d$ to indicate the dependence on the initial state of the process $X$. 

Below we consider an optimal stopping game between two players (Player~1 and Player~2), 
each of which does not know whether the other player is {\em active} or {\em passive}. 
By an `active player' we mean a player who actively participates in the stopping game by choosing a (randomised) stopping time; a `passive player' is not participating in the game, and will never stop. 
Mathematically, Player~$i$ uses the random variable $\theta_i$, $i=1,2$ to model whether the other player, Player~$(3-i)$, is active or passive. Specifically, Player~$1$ has active competition if $\theta_1=1$ and no competition if $\theta_1=0$; Player~$2$ has active 
competition if $\theta_2=1$ and no competition if $\theta_2=0$.

To describe the strategies in this type of game, we first introduce the notion of a randomised stopping time (see \cite{LS}, \cite{TV} and \cite{DEG}).

\begin{definition}{\bf (Randomised stopping times.)}\label{def:rst}
Let $\A$ be the set of right-continuous non-decreasing $\F^X$-adapted processes 
$\Gamma=\{\Gamma_t;0-\leq t<\infty\}$ satisfying 
$\Gamma_{0-}=0$ and $\Gamma_t\leq 1$ for all $t\geq 0$. Let 
$U$ be a random variable, which is independent of $X$, with $U\sim\text{Uniform}(0,1)$.
A $U$-{\em randomised stopping time} $\gamma$ is a random variable of the form 
\begin{align}\label{gamma}
\gamma=\inf\{t\geq 0:\Gamma_t>U\}
\end{align}
for some $\Gamma\in \A$. Furthermore, we say that $\gamma$ in \eqref{gamma} is {\em generated} by $\Gamma$.
\end{definition}

The collection of $U_i$-randomised stopping times is denoted $\T^R_i$ for $i=1,2$ and 
$U_i$ introduced in (c) above. Moreover, for future reference we also introduce stopping times
\begin{align}\label{eq:tau-u}
\gamma_u:=\inf\{t\geq 0:\Gamma_t>u\}, \quad u\in[0,1],
\end{align}
with $\Gamma\in\A$. Finally, we note that $\Delta\Gamma_t:=\Gamma_t-\Gamma_{t-}$ for $t\ge 0$ and $\Gamma\in\A$.

Let $g:\R^d\to[0,\infty)$ be a continuous function such that  
\begin{align}\label{eq:gpos}
\sup_{x\in\R^d}g(x)>0.
\end{align}
Our game is now specified as follows. Each player chooses a randomised stopping strategy: Player 1 chooses $\tau\in\T^R_1$ and Player 2 chooses $\gamma\in\T^R_2$. The payoff for Player~$1$ at time $\tau$ is 
\begin{eqnarray}\label{R1}
R_1(\tau,\gamma) &:=& \left(g(X_\tau)1_{\{\tau<\hat\gamma\}}  
+\frac{1}{2} g(X_\tau) 1_{\{\tau= \hat\gamma\}}
\right)1_{\{\tau<\infty\}},
\end{eqnarray}
where 
\begin{equation}
\label{hatgamma}
\hat\gamma:=\gamma 1_{\{\theta_1=1\}}+\infty 1_{\{\theta_1=0\}}.
\end{equation}
Similarly, at time $\gamma$, Player~2 receives the amount
\begin{eqnarray}\label{R2}
R_2(\tau,\gamma) &:=& \left(g(X_\gamma)1_{\{\gamma<\hat\tau\}}  
+ \frac{1}{2} g(X_\gamma) 1_{\{\gamma= \hat\tau\}}\right)
1_{\{\gamma<\infty\}},
\end{eqnarray}
where 
\[\hat\tau:=\tau 1_{\{\theta_2=1\}}+\infty 1_{\{\theta_2=0\}}.\]
Here the half in the second term indicates that, in the case when both players stop simultaneously, the payoff is split evenly between them. Alternative specifications of the payoff in case of simultaneous stopping are clearly possible and the methods that we develop in this paper may be extended to cover various situations. However, in the interest of a simple notation, we refrain from addressing those extensions here.

The aim of Player~1 (2) is to choose the randomised stopping time $\tau$ ($\gamma$) to maximize her/his expected discounted payoff, defined by
\[\J_1(\tau,\gamma;p_1,x):=\E_x[e^{-r\tau}R_1(\tau,\gamma)]\]
\[(\J_2(\tau,\gamma;p_2,x):=\E_x[e^{-r\tau}R_2(\tau,\gamma)]),\]
where the discount rate $r\geq 0$ is a given constant and $x\in\R^d$. 
By independence of $\theta_i$ from $\F^X$ we can rewrite Player 1's expected payoffs as 
\begin{align}\label{eq:Ji0}
\J_1(\tau,\gamma;p_1,x) \!=& (1-p_1)\E_x[e^{-r\tau}g(X_\tau)1_{\{\tau<\infty\}}]\\
&\!+\!p_1\E_x[e^{-r\tau}g(X_\tau) 1_{\{\tau< \gamma\}}]+ \frac{p_1}{2}\E_x[e^{-r\tau}g(X_\tau) 1_{\{\tau= \gamma<\infty\}}].\notag
\end{align}
A similar expression holds for Player 2's expected payoff $\J_2$, upon replacing $p_1$ with $p_2$ and swapping $\tau$ and $\gamma$ in the obvious way.

\begin{definition}\label{def:NE}
{\bf (Nash equilibrium.)}
Given $x\in\R^d$ and $p_i\in(0,1]$, $i=1,2$, a pair $(\tau^*,\gamma^*)\in\T^R_1\times\T^R_2$ is a Nash equilibrium if 
\[\J_1(\tau,\gamma^*;p_1,x)\leq \J_1(\tau^*,\gamma^*;p_1,x)\]
and 
\[\J_2(\tau^*,\gamma;p_2,x)\leq \J_2(\tau^*,\gamma^*;p_2,x)\]
for all pairs $(\tau,\gamma)\in\T^R_1\times\T^R_2$.
Given an equilibrium pair $(\tau^*,\gamma^*)\in\T^R_1\times\T^R_2$ we define the equilibrium payoffs as
\begin{align}\label{eq:eqvalue}
v_i(p_i,x):=\J_i(\tau^*,\gamma^*;p_i,x),\qquad\text{for $i=1,2$}.
\end{align}
\end{definition}

\begin{remark}
In the formulation above, each player is an active player, but is uncertain whether the other player is active or passive. An alternative formulation would be to stipulate that Player~$\mathrm i$ receives her/his payoff only if the
Bernoulli random variable $\theta_{3-i}$ takes the value 1 (i.e.~if Player $\mathrm i$ is indeed active). 
More precisely, the payoff to Player 1 becomes
\begin{eqnarray*}
\hat R_1(\tau,\gamma) &:=& \Big(g(X_\tau)1_{\{\tau<\hat\gamma\}}  
+\frac{1}{2} g(X_\tau) 1_{\{\tau= \hat\gamma\}}\Big)1_{\{\hat\tau<\infty\}}\\
&=& 1_{\{\theta_2=1\}}R_1(\tau,\gamma)
\end{eqnarray*}
(and a similar expression for Player 2). Then, by independence, the corresponding expected discounted payoffs
\begin{equation*}
\hat{\J}_i(\tau,\gamma;p_i,x):=\E_x[e^{-r\tau}\hat R_i(\tau,\gamma)],\qquad\text{$i=1,2$,}
\end{equation*}
satisfy $\hat{\J}_i(\tau,\gamma;p_i,x) =p_{3-i}\J_i(\tau,\gamma;p_i,x)$. In particular, a pair of strategies $(\tau,\gamma)\in\mathcal T^R_1\times \mathcal T^R_2$
is a Nash equilibrium for the game with expected payoffs 
$\hat{\mathcal J}_i$ if and only if it is a Nash equilibrium for the game with expected payoffs $\mathcal J_i$, so the two formulations are equivalent from a game-theoretic perspective.
\end{remark}

\begin{remark}
Our set-up and the results in this paper could be extended to consider a continuous strong Markov process on a domain $\I$, provided that the boundary behaviour of $X$ at $\partial\I$ is carefully accounted for. Such extension is straightforward when $X$ is a regular 1-dimensional diffusion on an open interval $\I\subseteq\R$ with natural boundaries. 
\end{remark}


\section{Some useful observations on the game's payoffs}\label{sec:obs}

In this section we make a few general observations that provide some intuition for the structure of the equilibrium payoffs of the two players.

\subsection{Bounds in terms of the single-player game}

Denote by
\begin{equation}
\label{V}
V(x):=\sup_{\tau\in\mathcal T} \E_x[e^{-r\tau}g(X_\tau)1_{\{\tau<\infty\}}]
\end{equation}
the value function of the corresponding single-player stopping game. (For a single-player game there is no need for randomisation, and 
we remark that, accordingly, the supremum in \eqref{V} is taken over stopping times.) For future reference we denote
\begin{equation}
 \label{tau}
\tau^*_V:=\inf\{t\geq 0:V(X_t)=g(X_t)\}.
\end{equation}

From now on, we make the following standing assumption.
\begin{assumption}\label{ass:cont}
The integrability condition 
\begin{equation}
\label{integrability}
\E_x\left[\sup_{t\geq 0}e^{-rt}g(X_t)\right]<\infty,\qquad x\in\R^d,
\end{equation}
holds, and the function $V:\R^d\to[0,\infty)$ is continuous. 
Furthermore, 
\begin{equation}
\label{limsup}
\limsup_{t\to\infty}e^{-rt}V(X_t)1_{\{\tau^*_V=+\infty\}}=0.
\end{equation}
\end{assumption}

\begin{remark}\label{rem:OS}
Continuity of $V$ is known to hold in virtually all examples addressed in the literature on optimal stopping.
The integrability condition \eqref{integrability} and condition
\eqref{limsup} guarantee that the stopping time $\tau^*_V$
is an optimal stopping time in \eqref{V}, see for example \cite{PS}. Furthermore, the process
\[Y_t:=e^{-rt}V(X_t),\qquad t\in[0,+\infty]\]
with $Y_\infty:=0$ is a supermartingale, and 
\[M_t:=e^{-r(t\wedge \tau^*_V)}V(X_{t\wedge \tau^*_V}),\qquad t\in[0,+\infty)\]
is a uniformly integrable martingale. 

While \eqref{limsup} is not necessary for the methods that we develop in the rest of the paper, it is a convenient assumption that helps us keep the exposition simple and avoid lengthy localisation arguments in the proof of Theorem~\ref{thm:NE}. 
\end{remark}

It is clear from \eqref{eq:Ji0} that $\J_1(\tau,\gamma;p_1,x)\le \E_x[e^{-r\tau}g(X_\tau)1_{\{\tau<\infty\}}]$ for any $(\tau,\gamma)\in\mathcal T^R_1\times\T^R_2$. Moreover, for $\tau\in\T_1^R$ we have
\begin{eqnarray}
\label{rand/st}
\E_x[e^{-r\tau}g(X_\tau)1_{\{\tau<\infty\}}] &=& \int_0^1 \E_x[e^{-r\tau}g(X_\tau)1_{\{\tau<\infty\}}\vert U_1=u]\,du\\
&=& \notag
\int_0^1 \E_x[e^{-r\tau_u}g(X_{\tau_u})1_{\{\tau_u<\infty\}}]\,du\\
&\leq& \notag V(x),
\end{eqnarray}
where $\tau_u\in\T$ is defined as in \eqref{eq:tau-u}.
Thus, allowing for randomisation in the single-player game does not increase the value and we have the following upper bound 
\begin{align}\label{eq:upb}
\sup_{\tau\in\T^R_1}\J_1(\tau,\gamma;p_1,x)\le \sup_{\tau\in\T^R_1}\E_x[e^{-r\tau}g(X_\tau)1_{\{\tau<\infty\}}]= V(x)
\end{align} 
for any $\gamma\in\T^R_2$. The same bound obviously holds for Player 2's expected payoff.

Next we obtain useful lower bounds. First, for any given $\gamma\in\T^R_2$ Player 1 could choose $\tau=\tau^*_V$ defined in \eqref{tau}. Therefore, using \eqref{eq:Ji0} we get
\begin{align}\label{eq:lob}
\sup_{\tau\in\T^R_1}\inf_{\gamma\in\T^R_2}\J_1(\tau,\gamma;p_1,x) &\ge 
(1-p_1)\E_x\left[e^{-r\tau^*_V}g(X_{\tau^*_V})1_{\{\tau^*_V<\infty\}}\right]\\
\notag
&=(1-p_1)V(x).
\end{align} 
Second, using that $\tau=0$, $\P$-a.s., is an admissible stopping rule gives
\begin{eqnarray}\label{eq:b1}
&&\sup_{\tau\in\T^R_1}\inf_{\gamma\in\T^R_2}\J_1(\tau,\gamma;p_1,x)\\
\notag
&\ge& \inf_{\gamma\in\T^R_2}\left\{(1-p_1)g(x)+p_1 g(x)\P(\gamma>0)+\frac{p_1}{2}g(x)\P(\gamma=0)\right\}\\
&=& \left(1-\frac{p_1}{2}\right)g(x)\notag
\end{eqnarray}
for any $x\in\R^d$. Summarising \eqref{eq:upb}-\eqref{eq:b1}, we have
\begin{equation}
\label{eq:uplob1}
 \max\left\{(1-p_1)V(x), (1-p_1/2)g(x)\right\}\leq 
 \sup_{\tau\in\T^R_1}\inf_{\gamma\in\T^R_2}\J_1(\tau,\gamma;p_{1},x)\leq V(x),
\end{equation}
and similar arguments also lead to 
\begin{align}\label{eq:uplob}
\max\left\{(1-p_2)V(x), (1-p_2/2)g(x)\right\}\le\sup_{\gamma\in\T^R_2}
\inf_{\tau\in\T^R_1}\J_2(\tau,\gamma;p_{2},x)\le V(x).
\end{align} 

It follows from \eqref{eq:uplob1}-\eqref{eq:uplob} that 
\begin{equation}
 \label{U}
 V_0(p_i,x):=\max\{(1-p_i)V(x),\, (1-p_i/2)g(x)\}
\end{equation}
is a `safety level' for Player~$i$.
More precisely, the equilibrium values (see \eqref{eq:eqvalue}) of any Nash equilibrium $(\tau^*,\gamma^*)\in\T^R_1\times\T^R_2$ have to satisfy $v_{i}(p_i,x)\geq V_0(p_i,x)$.

\subsection{A more explicit form of the expected payoffs}
We next present a convenient way to rewrite the problem, which will inform our construction of an equilibrium
in Section~\ref{sec:construction}. As anticipated, we expect that equilibria be found in randomised strategies associated to increasing processes. Forgetting for a moment equilibrium considerations, however, if Player $i$ plays a randomised stopping time, then by an argument as in \eqref{rand/st} above, Player $3-i$'s optimal response can be found in the class of normal stopping times. In particular, letting $\tau\in\T^R_1$ and $\gamma\in\T^R_2$ be arbitrary one has
\begin{align}
\label{eq:ran1}&\sup_{\zeta\in\T^R_1}\J_1(\zeta,\gamma;p_1,x)=\sup_{\zeta\in\T}\J_1(\zeta,\gamma;p_1,x),\\
\label{eq:ran2}&\sup_{\zeta\in\T^R_2}\J_2(\tau,\zeta;p_2,x)=\sup_{\zeta\in\T}\J_2(\tau,\zeta;p_2,x).
\end{align}

Although this argument does not in general lead to an equilibrium, it suggests that, given a randomised strategy of Player $i$'s, we can restrict our attention to the expected payoff of Player $3-i$ evaluated at normal stopping times. In the next proposition we give explicit formulae for such payoffs.

\begin{proposition}
\label{altform}
For $i=1,2$ let $\Gamma^{i}\in\A$, and let $\tau\in\T^R_1$ and $\gamma\in\T^R_2$ be generated by $\Gamma^1$ and $\Gamma^2$, respectively. 

For any $\zeta\in\T$ and $x\in\R^d$ we have
\begin{eqnarray}\label{eq:J1}
\J_1(\zeta, \gamma;p_1,x) &=&  (1-p_1)\E_x\left[e^{-r\zeta}g(X_\zeta)1_{\{\zeta<+\infty\}}\right] \\
&& +p_1\E_x\left[e^{-r\zeta}g(X_\zeta)(1-\Gamma^{2}_\zeta)1_{\{\zeta<+\infty\}}\right]\notag\\
&&+\frac{p_1}{2} \E_x\left[e^{-r\zeta}g(X_\zeta)\Delta\Gamma^2_\zeta 1_{\{\zeta<+\infty\}}\right]\notag
\end{eqnarray}
and
\begin{eqnarray}\label{eq:J2}
\J_2(\tau,\zeta;p_2,x) &=&  (1-p_2)\E_x\left[e^{-r\zeta}g(X_\zeta)1_{\{\zeta<+\infty\}}\right] \\
&& +p_2\E_x\left[e^{-r\zeta}g(X_\zeta)(1-\Gamma^{1}_\zeta)1_{\{\zeta<+\infty\}}\right]\notag\\
&&+\frac{p_2}{2} \E_x\left[e^{-r\zeta}g(X_\zeta)\Delta\Gamma^1_\zeta 1_{\{\zeta<+\infty\}}\right].\notag
\end{eqnarray}
\end{proposition}

\begin{proof}
By symmetry, it is sufficient to consider the case $i=2$. Thus we let $\gamma\in\T^R_2$ be generated by $\Gamma^2\in\A$, and we want to show that \eqref{eq:J1} holds. 

Recall from \eqref{eq:Ji0} that
\begin{align}\label{eq:Ji0'}
\J_1(\zeta,\gamma;p_1,x) =& (1-p_1)\E_x[e^{-r\zeta}g(X_\zeta)1_{\{\zeta<\infty\}}]+p_1\E_x[e^{-r\zeta}g(X_\zeta) 1_{\{\zeta< \gamma\}}]\\
&+ \frac{p_1}{2}\E_x[e^{-r\zeta}g(X_\zeta) 1_{\{\zeta= \gamma<\infty\}}],\notag
\end{align}
and the first term on the right-hand side of \eqref{eq:Ji0'} is identical to the first term of \eqref{eq:J1}. 
For the second term, by Definition \ref{def:rst} we have
\begin{eqnarray*}
\E_x[e^{-r\zeta}g(X_\zeta)1_{\{\zeta<\gamma\}}]&=&\E_x\Big[e^{-r\zeta}g(X_\zeta)\P_x\big(\zeta<\gamma\big|\F^X_\zeta\big)1_{\{\zeta<\infty\}}\Big]\\
&=&\E_x\Big[e^{-r\zeta}g(X_\zeta)(1-\Gamma^2_\zeta)1_{\{\zeta<\infty\}}\Big],
\end{eqnarray*}
where for the final expression we used that for any $t\ge 0$ it holds
\begin{align}\label{eq:incl0}
\{\Gamma^2_t<U_2\}\subset\{\gamma>t\}\subseteq \{\Gamma^2_t\le U_2\}.
\end{align}

The final term in \eqref{eq:Ji0'} gives
\begin{align}\label{eq:Delta}
\E_x[e^{-r\zeta}g(X_\zeta)1_{\{\zeta=\gamma<+\infty\}}]=\E_x\Big[e^{-r\zeta}g(X_\zeta)\P_x\big(\zeta=\gamma\big|\F^X_\zeta\big)1_{\{\zeta<+\infty\}}\Big].
\end{align}
Now we notice that
\begin{eqnarray*}
\P_x\big(\zeta=\gamma\big|\F^X_\zeta\big)&=&\P_x\big(\gamma \ge \zeta \big|\F^X_\zeta\big)-\P_x\big(\gamma > \zeta \big|\F^X_\zeta\big)\\
&=&\lim_{\delta\to 0}\P_x\big(\gamma> \zeta -\delta \big|\F^X_\zeta\big)-(1- \Gamma^2_\zeta)\\
&=&\lim_{\delta\to 0}(1-\Gamma^2_{\zeta-\delta})-(1- \Gamma^2_\zeta)=\Delta\Gamma^2_\zeta
\end{eqnarray*}
which, combined with \eqref{eq:Delta}, concludes our proof.
\end{proof}

\section{Construction of a candidate strategy}
\label{sec:construction}

\subsection{Adjusted beliefs}

In order to find an equilibrium for the game we study the evolution (during the game) of the players' {\em beliefs} regarding the existence of their opponent. In other words, if $\gamma\in\T^R_2$ is generated by $\Gamma^2\in\A$, then Player 1 
dynamically evaluates the conditional probability of Player 2 being active as
\begin{eqnarray}\label{Pi1}
\Pi^1_t&:=&\P(\theta_1=1|\F^X_t,\hat\gamma>t)=
\frac{\P(\theta_1=1|\F^X_t)\P(\hat\gamma>t\vert \F^X_t, \theta_1=1)}{\P(\hat\gamma>t|\F^X_t)}\\
&=&\frac{p_1\P(\gamma>t|\F^X_t)}{1-p_1+p_1\P(\gamma>t|\F^X_t)}=\frac{p_1(1-\Gamma^2_t)}{1-p_1\Gamma^2_t}\notag
\end{eqnarray}
provided $p_1\in(0,1)$, where we recall that $\hat\gamma=\gamma 1_{\{\theta_1=1\}}+\infty 1_{\{\theta_1=0\}}$ as in \eqref{hatgamma}.
Here we used independence of $\gamma$ and $\theta_1$ in the third equality, 
and in the final equality we used that $\P(\gamma>t|\F^X_t)=1-\Gamma^2_t$ because of \eqref{eq:incl0}.

Likewise, if $\tau\in\T^R_1$ is generated by $\Gamma^1\in\A$, then Player 2 evaluates the conditional probability of Player 1 being active as
\begin{align}\label{Pi2}
\Pi^2_t:=&\P(\theta_2=1|\F^X_t,\hat\tau>t)=\frac{p_2(1-\Gamma^1_t)}{1-p_2\Gamma^1_t}
\end{align}
provided $p_2\in(0,1)$.

Note that there is a one-to-one correspondence between $\Gamma^{3-i}$ and $\Pi^i$, for $i=1,2$. In fact, 
\begin{equation}
 \label{GammaPi}
 \Gamma^{3-i}_t=\frac{p_i-\Pi^{i}_t}{p_i(1-\Pi^{i}_t)},\qquad i=1,2.
\end{equation}
Furthermore, the equality 
\begin{equation}
\label{alg}
(1-p_i)=(1-p_i\Gamma^{3-i}_t)(1-\Pi^i_t),\qquad i=1,2,
\end{equation}
holds. We will see below how these formulae become a key ingredient in the analysis of our game. 

\subsection{A reflected adjusted belief process}

We now give a construction of a right-continuous process $(Z,X)$ that is bound to evolve in a particular subset of $[0,1)\times\R^d$ (denoted $\overline\C$ below). The process $Z$ will play the role of an adjusted belief process $\Pi$ corresponding to a suitably constructed increasing process $\Gamma$ (obtained from $Z$ by the relation \eqref{GammaPi}), see Lemma~\ref{lem:gamma} below.

We first introduce the disjoint sets 
\begin{align}
\label{eq:Ci}&\OC:=\{(p,x)\in(0,1)\times\R^d:(1-p)V(x)\geq g(x)\}\\
\label{eq:Si}&\C':=\{(p,x)\in(0,1)\times\R^d:(1-p/2)g(x)<(1-p)V(x)< g(x)\}\\
\label{eq:Spi}&\S:=\{(p,x)\in(0,1)\times\R^d:(1-p)V(x)\leq (1-p/2)g(x)\}
\end{align}
and note that $\OC\cup\C'\cup\S=(0,1)\times \R^d$. 
Since $p\mapsto(1-p)V(x)$ is non-increasing, we have that 
if $(p,x)\in\OC$, then $(0,p)\times\{x\}\subseteq\OC$. Define
\[b(x):=\inf\{p\in[0,1]:(1-p)V(x)\leq g(x)\}\in[0,1].\]
Likewise, $p\mapsto(1-p)V(x)-(1-p/2)g(x)$ is non-increasing, so
if $(p,x)\in\C'$ then  $(b(x),p)\times\{x\}\subseteq\C'$, and we define
\[c(x):=\inf\{p\in[0,1]:(1-p)V(x)\leq(1-p/2) g(x)\}\in[0,1].\]

Then $b(x)\le c(x)$ for $x\in\R^d$ and we have
\begin{align}
\label{eq:C2}&\OC=\{(p,x)\in(0,1)\times\R^d:p\leq b(x)\},\\
\label{eq:C3}&\C'=\{(p,x)\in(0,1)\times\R^d:b(x)< p<c(x)\}
\end{align}
and 
\[\S=\{(p,x)\in(0,1)\times\R^d:c(x)\leq p\}.\]


\begin{lemma}
\label{cont}
The functions $b,c:\R^d\to[0,1]$ are continuous. Moreover, 
\begin{itemize}
\item[(i)]
$b(x)=c(x)=0 \iff V(x)=g(x);$
\item[(ii)]
$b(x)=c(x)=1 \iff g(x)=0.$
\end{itemize}
\end{lemma}

\begin{proof}
We first note that $V>0$ on $\R^d$ since $X$ is regular and \eqref{eq:gpos} holds. Moreover, $V>g/2$ on $\R^d$
(since if $g(x)=0$ then $V(x)>0=g(x)/2$ and if $g(x)>0$ then $V(x)\geq g(x)>g(x)/2$).
Then strict monotonicity and continuity of $p\mapsto (1-p)V(x)$ imply that for any 
$x\in\R^d$ the value of $b(x)$ is uniquely determined by  
\begin{align}\label{eq:bx}
(1-b(x))V(x)=g(x), 
\end{align}  
and strict monotonicity and continuity of $p\mapsto(1-p)V(x)-(1-p/2)g(x)$
(since $V>g/2$ on $\R^d$) imply that $c(x)$ is uniquely determined by 
\begin{align}\label{eq:cx}
(1-c(x))V(x)=(1-p/2)g(x).
\end{align}  
Consequently,
\begin{align}\label{eq:b(x)}
b(x)=1-\frac{g(x)}{V(x)}
\end{align}
and
\begin{align}\label{eq:c(x)}
c(x)=\frac{V(x)-g(x)}{V(x)-g(x)/2}.
\end{align}
Continuity of $b$ and $c$ thus follow from Assumption~\ref{ass:cont}. Finally, (i) and (ii) follow from the fact that $V>0$ and $V\ge g$.
\end{proof}

\begin{proposition}\label{prop:Z}
Let $(p,x)\in\OC$ be given and fixed and define $\P_x$-a.s.~the process
\begin{align}\label{eq:SIZ}
Z_t:=p\wedge \inf_{0\le s\le t}b(X_s).
\end{align}
Then $\P_x$-a.s.~
\begin{itemize}
\item[(i)] $Z$ is non-increasing and continuous;
\item[(ii)] $(Z_t,X_t)\in\OC$ for all $t\ge 0$;
\item[(iii)] we have 
\begin{align}\label{eq:SK}
\ud Z_t=1_{\{(1-Z_{t})V(X_t)=g(X_t)\}}\ud Z_t
\end{align} 
as (random) measures.
\end{itemize} 
\end{proposition}

\begin{proof}
Since $x\mapsto b(x)$ and $t\mapsto X_t$ are continuous, it is immediate to verify that (i) holds. Moreover, by definition $Z_t\le b(X_t)$ for all $t\ge 0$, hence implying (ii).

In order to check that (iii) holds we fix $\omega\in\Omega$ (outside of a null set) and take an arbitrary $t\ge 0$ such that 
\begin{align*}
(1-Z_t(\omega))V(X_t(\omega))>g(X_t(\omega)).
\end{align*}
By definition of $b$, the strict inequality above implies $Z_t(\omega)<b(X_t(\omega))$. Then by continuity of $t\mapsto b(X_t)$ there exists $\delta_{t,\omega}>0$ such that $Z_t(\omega)<b(X_{t+s}(\omega))$
for all $s\in(0,\delta_{t,\omega})$. Hence
\begin{align*}
Z_{t+s}(\omega)=Z_t(\omega)\wedge\inf_{0<u\le s}b(X_{t+u}(\omega))=Z_t(\omega),\quad\text{for all $s\in(0,\delta_{t,\omega})$}.
\end{align*}
The latter equation implies $\ud Z_t(\omega)=0$ as needed in \eqref{eq:SK}. 
\end{proof}

We point out that $Z_t=Z^{p,x}_t$ depends on the initial point $(p,x)$.
Associated to $Z$ above we now construct a process $\Gamma\in\A$, which will be used to generate the randomised stopping times for the Nash equilibrium in the game. In particular, in the next lemma we will recall the belief process introduced in \eqref{Pi1} and \eqref{Pi2}.

\begin{lemma}\label{lem:gamma}
For $(p,x)\in\OC$, define the process
\begin{align}\label{eq:Gpx}
\Gamma_t:=
\Gamma^{p,x}_t:=\frac{p-Z^{p,x}_t}{p(1-Z^{p,x}_t)},\:\:t\ge 0,
\end{align}
with $Z^{p,x}=Z$ as in \eqref{eq:SIZ}. Then $\Gamma=\Gamma^{p,x}$ is continuous, $\Gamma\in\A$ and the adjusted belief process generated by $\Gamma$, i.e. the process
\begin{align}\label{eq:PiG}
\Pi^{\Gamma}_t:=\frac{p(1-\Gamma_t)}{1-p\Gamma_t},
\end{align}
satisfies $\Pi^\Gamma =Z$, $\P_x$-a.s. Hence $(\Pi^\Gamma_t,X_t)\in\overline\C$ for all $t\ge 0$, $\P_x$-a.s.
Furthermore, if $\tau\in\T^R_1$ is generated by $\Gamma$, then $\tau\leq \tau^*_V$.
\end{lemma}

\begin{proof}
For the first claim, by comparing \eqref{eq:Gpx} and \eqref{eq:PiG} it is immediate to check that indeed $\Pi^\Gamma=Z$. It remains to verify that $\Gamma\in\A$. This follows since $\Gamma$ is continuous, by continuity of $Z$, and it is non-decreasing because both the process $Z$ and the mapping $z\mapsto (p-z)/[p(1-z)]$ are non-increasing. 

For the second claim we have by Lemma~\ref{cont} and \eqref{eq:SIZ} that
\begin{align}\label{eq:taugamma}
\tau^*_V =& \inf\{t\geq 0:V(X_t)=g(X_t)\} \\
=& \inf\{t\geq 0: b(X_t)=0\}\notag\\
=& \inf\{t\geq 0: Z_t=0\} = \inf\{t\geq 0:\Gamma_t=1\},\notag
\end{align}
from which it follows that $\tau_u=\inf\{t\geq 0:\Gamma_t>u\}\leq \tau^*_V$ for $u\in[0,1)$. Consequently, 
$\tau\leq \tau^*_V$, $\P_x$-a.s.
\end{proof}


\section{Construction of Nash equilibria}
\label{sec:NE}

With no loss of generality, we assume that $p_2$ dominates $p_1$, i.e. $0\leq p_1\leq p_2\leq 1$. We first comment on the two extreme cases $p_1=0$ and $p_1=1$.

If $p_1=1$, then also $p_2=1$ (since $p_1\leq p_2$), so both players are certain that the other player is active. 
In this case it is clear that immediate stopping, i.e. $(\tau^*,\gamma^*):=(0,0)$, provides a Nash equilibrium, and the corresponding equilibrium value for each player is $g(x)/2$.

If $p_1=0$, then Player~1 is certain that Player~2 is not active, and will consequently play the 
optimal strategy $\tau^*=\tau_V^*$ from the single-player game. If also $p_2=0$, then
the optimal response for Player~2 is also to use $\gamma^*=\tau^*_V$, and the corresponding equilbrium value is given by $V(x)$ for each player. If instead $p_2>0$, then the situation is slightly more involved: the optimal response for Player 2 would be to preempt Player 1 by stopping just before $\tau^*_V$ (at $\tau^*_V-$); however, this is not a randomised stopping time, and there is no equilibrium in this setting.

Throughout the rest of this section we impose the condition 
\[
0<p_1\leq p_2\leq 1\quad\text{and}\quad p_1<1,
\]
thus excluding the two cases discussed above. Then we construct Nash equilibria for our game using results from the previous two sections.

One interesting feature is that our construction naturally splits into three scenarios. 
Namely, the qualitative properties of the equilibrium that we obtain depend on 
which of the three regions $\OC$, $\C'$ and $\S$ that $(p_1,x)$ belongs to. Equally remarkable seems to be the fact that the initial value of $p_2\in[p_1,1]$ does not matter for the qualitative aspects of our construction. We will comment more extensively on this feature in Section \ref{sec:comment} below. In the same section we will also discuss the heuristics that led us to the construction of the equilibria detailed in Theorems~\ref{thm:NE} and \ref{thm:NE3} below.

\subsection{Equilibrium for $(p_1,x)\in\OC\cup\C'$}\label{sec:NEsym} 
Throughout this section we consider $(p_1,x)\in\OC\cup\C'$. We first define
\begin{align}\label{eq:Gq}
\Gamma^*_0:=\frac{2}{p_1}\left(1-\frac{(1-p_1)V(x)}{g(x)}\right)^+\quad\text{and}\quad q_1:=\frac{p_1(1-\Gamma^*_0)}{1-p_1\Gamma^*_0}
\end{align}
and note that 
\begin{equation}
\label{p1q1}
(1-p_1\Gamma_0^*)(1-q_1)=1-p_1.
\end{equation}

\begin{lemma}\label{lem:5.1}
For $(p_1,x)\in\OC\cup\C'$ we have 
\begin{itemize}
\item[(i)]
$0\leq \Gamma^*_0< 1$,
\item[(ii)]
$0< q_1\leq b(x)$.
\end{itemize}
Moreover $0<q_1<b(x)$ for $(p_1,x)\in\C'$.
\end{lemma}

\begin{proof}
If $p_1\leq b(x)$ then $\Gamma^*_0=0$ and $q_1=p_1\in(0,b(x)]$, so (i)-(ii) hold.

If $p_1\in(b(x),c(x))$ then 
\[\Gamma^*_0=\frac{2}{p_1}\left(1-\frac{(1-p_1)V(x)}{g(x)}\right) \in(0,1)\]
so (i) holds and $q_1>0$.
Furthermore, the first equation in \eqref{eq:Gq} reads
\begin{align}\label{eq:use0}
(1-p_1)V(x)=g(x)\left(1-\frac{p_1}{2}\Gamma^*_0\right).
\end{align}
Combining the latter and \eqref{p1q1} we get
\begin{align}\label{eq:jump}
(1-q_1)V(x)=\frac{1-\frac{p_1}{2}\Gamma^*_0}{1-p_1\Gamma^*_0}g(x)>g(x)
\end{align}
since it must be $g(x)>0$ by (ii) in Lemma~\ref{cont} (otherwise $p_1\le b(x)=1$). Consequently, $(q_1,x)$ lies in the interior of the set $\OC$, so $q_1<b(x)$, which implies (ii).
\end{proof}

Recall the processes $Z$ and $\Gamma$ introduced in Proposition \ref{prop:Z} and Lemma~\ref{lem:gamma} along with the adjusted belief process $\Pi^\Gamma$. Define $\Gamma^{2,*}$ as 
\begin{align}\label{eq:G2*}
\Gamma^{2,*}_t:=\Gamma^*_0+(1-\Gamma^*_0)\Gamma^{q_1,x}_t,\quad t\geq 0,
\end{align}
with $\Gamma^*_0$ as in \eqref{eq:Gq} and $\Gamma^{q_1,x}$ as in Lemma~\ref{lem:gamma} but with $p=q_1$,  let 
\begin{align}\label{eq:G001*}
\Gamma^{1,*}_t:=\tfrac{p_1}{p_2}\Gamma^{2,*}_t 1_{\{t<\tau^*_V\}} + 1_{\{t\ge\tau^*_V\}},
\end{align}
and let 
\[\Pi^{1,*}_t:=\frac{p_1(1-\Gamma^{2,*}_t)}{1-p_1\Gamma^{2,*}_t}\]
be the adjusted belief of Player 1 corresponding to $\Gamma^{2,*}$ (recall \eqref{Pi1}).

\begin{proposition}
\label{propDelta}
For $(p_1,x)\in\OC\cup\C'$ the following properties hold:
\begin{itemize}
\item[(i)]
$(\Gamma^{1,*},\Gamma^{2,*})\in\A^2$;
\item[(ii)]
we have $\P_x$-a.s.
\[\Delta\Gamma^{2,*}_t=\left\{\begin{array}{ll}
\Gamma^*_0, & t=0,\\
0, &t>0;\end{array}\right.\]
\item[(iii)]
we have $\P_x$-a.s.
\[\Delta\Gamma^{1,*}_t=\left\{\begin{array}{ll}
\frac{p_1}{p_2}\Gamma^*_0, & t=0,\\
1-p_1/p_2, & t=\tau^*_V<\infty,\\
0, & \mbox{otherwise};\end{array}\right.\]
\item[(iv)]
$(\Pi^{1,*}_t,X_t)\in\overline{\C}$ for $t\ge 0$.
\end{itemize}
\end{proposition}

\begin{proof}
Recalling the expression in \eqref{eq:Gpx} for $\Gamma^{q_1,x}$ and that of $Z^{q_1,x}$ in \eqref{eq:SIZ} it is immediate to verify that $\Gamma^{q_1,x}\in[0,1]$ and it is continuous. Hence, $\Gamma^{2,*}\in[0,1]$ and it is right-continuous with its only possible jump at time zero. Thus $(\Gamma^{1,*},\Gamma^{2,*})\in\A^2$ so (i) holds. 

The expressions in (ii) and (iii) for 
$\Delta\Gamma^{2,*}_0$ and $\Delta\Gamma^{1,*}_0$ are immediate consequence of 
\eqref{eq:G2*} and \eqref{eq:G001*}. Furthermore, on $\{\tau^*_V<\infty\}$ we have $\Gamma^{q_1,x}_t=1$ for $t>
\tau^*_V$ by \eqref{eq:taugamma}, so continuity of $\Gamma^{q_1,x}$ implies that 
$\Gamma^{q_1,x}_{\tau^*_V-}=1$. Therefore 
\[\Delta\Gamma^{1,*}_{\tau^*_V}=1-\frac{p_1}{p_2}\Gamma^{2,*}_{\tau^*_V-}=1-\frac{p_1}{p_2},\]
which completes the proof of (iii).

For (iv), recalling the expression \eqref{Pi1} for the adjusted belief $\Pi^{1,*}$ of Player~1,
simple algebra yields
\begin{eqnarray}\label{eq:pis}
\Pi_t^{1,*}&=& \frac{p_1(1-\Gamma^{2,*}_t)}{1-p_1\Gamma^{2,*}_t}
=\frac{p_1(1-\Gamma^*_0)(1-\Gamma^{q_1,x}_t)}{1-p_1\Gamma^*_0-p_1(1-\Gamma^*_0)\Gamma^{q_1,x}_t}\\
&=&\frac{q_1(1-\Gamma^{q_1,x}_t)}{1-q_1\Gamma^{q_1,x}_t}=\Pi^{q_1,x}_t.\notag
\end{eqnarray}
Now \eqref{eq:pis} implies $(\Pi^{1,*}_t,X_t)\in\overline{\C}$ for $t\ge 0$, $\P_x$-a.s.~by virtue of 
Lemma~\ref{lem:gamma}, so (iv) holds.
\end{proof}

We are now ready to formulate our main result.

\begin{theorem}\label{thm:NE}
Let $(p_1,x)\in\overline \C\cup\C'$ be given and fixed (equivalently, $p_1<c(x)$), and 
define $(\Gamma^{1,*},\Gamma^{2,*})\in\A^2$ as in \eqref{eq:G001*} and \eqref{eq:G2*} above.
Then the strategy pair $(\tau^*,\gamma^*)$ generated by $(\Gamma^{1,*},\Gamma^{2,*})$~is a Nash equilibrium. Moreover, the equilibrium payoff for both players is $(1-p_1)V(x)$.
\end{theorem}

\begin{proof}
Consider the process $\{N_t,\,t\ge 0\}$ defined as follows: for fixed $(p_1,x)\in\overline\C\cup\C'$ let
\begin{align}\label{eq:N}
N_t&:=\left(1-\frac{p_1}{2}\Gamma^*_0\right)g(x)1_{\{t=0\}}+\tilde N_t 1_{\{t>0\}}
\end{align}
where
\[
\tilde N_t:=(1_{\{\theta_1=0\}} + 1_{\{\theta_1=1, U_2\ge\Gamma^*_0\}})(1-q_1)e^{-rt}V(X_t),
\]
and note that \eqref{eq:Gq}, \eqref{eq:use0} and \eqref{eq:N} imply
\begin{equation}
\label{N0}
N_0=g(x)\wedge \left((1-p_1)V(x)\right).
\end{equation}
With the notation $Y_t:=e^{-rt}V(X_t)$ introduced in Remark \ref{rem:OS}, recall that $\{Y_t,0\le t\le \infty\}$ with 
$Y_\infty:=0$ is a supermartingale. Thus, by \eqref{p1q1} and the optional sampling theorem, for any $\tau\in\T$, we have 
\begin{eqnarray}
\E_x\left[N_\tau\right]&=&N_01_{\{\tau=0\}}+1_{\{\tau>0\}}(1\!-\!p_1\Gamma^*_0)(1\!-\!q_1)\E_x\left[e^{-r\tau}V(X_{\tau})1_{\{\tau<\infty\}}\right]\notag\\
&=&
N_0 1_{\{\tau=0\}}+1_{\{\tau>0\}}(1\!-\!p_1)\E_x\left[e^{-r\tau}V(X_{\tau})1_{\{\tau<\infty\}}\right]\notag\\
&\le & N_01_{\{\tau=0\}}+1_{\{\tau>0\}}(1\!-\!p_1)V(x),\notag
\end{eqnarray}
where we have used that $\P_x(\tau=0)=1_{\{\tau=0\}}$ and $\P_x(\tau>0)=1_{\{\tau>0\}}$ by the $0-1$-Law. Finally, recalling \eqref{N0} we get
\begin{align}\label{eq:NJ00}
\E_x\left[N_\tau\right]\le (1-p_1)V(x),\qquad\text{for all $\tau\in\T$}.
\end{align}

Repeating the same steps as above and using that $M_t= Y_{t\wedge\tau_V^*}$ is a uniformly integrable martingale (Remark \ref{rem:OS}), we also obtain
\begin{align}\label{eq:NJ02}\begin{array}{ll}
\E_x\left[N_{\tau\wedge\tau_V^*}\right]= (1-p_1)V(x), & \text{for all $\tau\in\T$ with $\tau>0$ and}\\
 & \text{for $\tau=0$ if $\Gamma^*_0>0$}.\end{array}
\end{align}

Now we proceed in two steps.
\vspace{+4pt}

{\em Step 1}. (Optimality of $\tau^*$.) First we show that 
\begin{align}\label{eq:NJ}
\sup_{\tau\in\T}\J_1(\tau,\gamma^*;p_1,x)\le\sup_{\tau\in\T}\E_x[N_\tau]. 
\end{align}
Recalling Proposition \ref{altform} and that the only possible jump of $\Gamma^{2,*}$ occurs at $t=0$ 
(Proposition~\ref{propDelta}) we obtain, for any $\tau\in\T$,
\begin{align}\label{eq:NJ1}
\J_1(\tau,\gamma^*;p_1,x)
=&
\E_x\left[e^{-r \tau}g(X_\tau)(1-p_1\Gamma^{2,*}_\tau)1_{\{\tau<\infty\}}\right]+\frac{p_1}{2}g(x)\Gamma^*_01_{\{\tau=0\}}\\
=&\notag 1_{\{\tau>0\}}\E_x\left[e^{-r \tau}g(X_\tau)(1-p_1\Gamma^{2,*}_\tau)1_{\{\tau<\infty\}}\right]\notag\\
&+1_{\{\tau=0\}}g(x)\left(1-p_1\Gamma^*_0 +\frac{p_1}{2}\Gamma^*_0\right)\notag\\
=&1_{\{\tau=0\}} N_0+1_{\{\tau>0\}}\E_x\left[e^{-r \tau}g(X_\tau)(1-p_1\Gamma^{2,*}_\tau)1_{\{\tau<\infty\}}\right].\notag
\end{align}

By definition (see \eqref{eq:G2*}) we have
\begin{eqnarray*}
1-p_1\Gamma^{2,*}_\tau &=&1-p_1\Gamma^*_0-p_1(1-\Gamma^*_0)\Gamma^{q_1,x}_\tau\\
&=&(1-p_1\Gamma^*_0)(1-q_1\Gamma^{q_1,x}_\tau),
\end{eqnarray*}
where the second equality follows from \eqref{eq:Gq}. Substituting the last expression in \eqref{eq:NJ1} and using that $g(X_\tau)\le(1-\Pi^{1,*}_\tau)V(X_\tau)$ (see (iv) in Proposition \ref{propDelta}) we arrive at
\begin{align}\label{eq:NJ2}
\J_1(\tau,\gamma^*;p_1,x)
\le& 1_{\{\tau=0\}} N_0\\
&+1_{\{\tau>0\}}(1-p_1\Gamma^*_0)\E_x\left[e^{-r \tau}(1-q_1\Gamma^{q_1,x}_\tau)(1-\Pi^{1,*}_\tau)V(X_\tau)\right]\notag\\
=& \notag
1_{\{\tau=0\}} N_0+1_{\{\tau>0\}}(1-p_1\Gamma^*_0)(1-q_1)\E_x\left[e^{-r \tau}V(X_\tau)\right]\\
=& \E_x\left[N_\tau\right],\notag
\end{align}
where in the first equality we used $(1-q_1\Gamma^{q_1,x}_\tau)(1-\Pi^{1,*}_\tau)=(1-q_1)$, due to \eqref{eq:pis}, and the second one follows from \eqref{eq:N}. Thus we have proved \eqref{eq:NJ}. Combining the latter with \eqref{eq:NJ00} and \eqref{eq:ran1} gives
\begin{align}\label{eq:NJ3}
\sup_{\tau\in\T^R_1}\J_1(\tau,\gamma^*;p_1,x)\le (1-p_1)V(x).
\end{align} 

Now, take $\tau^*_u$ as in \eqref{eq:tau-u} but with $\Gamma^{1,*}$ instead of $\Gamma$. Then, since
$\Gamma^{1,*}_t=1$ for $t\geq \tau^*_V$ (by \eqref{eq:taugamma}), we have $\tau^*_u\le \tau^*_V$ for all $u\in[0,1)$. Furthermore, we claim that 
\begin{equation}
\label{eq:equal}
(1-\Pi^{1,*}_{\tau^*_u})V(X_{\tau^*_u})1_{\{\tau^*_u<\infty\}} =g(X_{\tau^*_u})1_{\{\tau^*_u<\infty\}}\,.
\end{equation}
Indeed, to see that \eqref{eq:equal} holds, first assume that $u\in[0,p_1/p_2)$. Then, recalling that $\Gamma^{2,*}$ is continuous for $t>0$ and $\Gamma^{2,*}<1$ for $t<\tau^*_V$, we have 
\[\tau^*_u=\inf\{t\geq 0:\Gamma^{1,*}_t>u\}=\inf\{t\geq 0:\Gamma^{2,*}_t>(p_2/p_1)u\}=\gamma^*_{(p_2/p_1)u}.\]
Hence \eqref{eq:equal} holds because $\tau^*_u$ is a time of increase of $\Gamma^{2,*}$ and $t\mapsto \ud \Gamma^{2,*}_t$ is supported on
\[
\{t\ge 0: (1-\Pi^{1,*}_t)V(X_t)=g(X_t)\}
\]
(see Proposition \ref{prop:Z} and Lemma \ref{lem:gamma}). On the other hand, if $u\in[p_1/p_2,1)$, then $\tau^*_u=\tau^*_V$, and 
\[
(1-\Pi^{1,*}_{\tau^*_u})V(X_{\tau^*_u})
=(1-\Pi^{1,*}_{\tau^*_V})V(X_{\tau^*_V})=V(X_{\tau^*_V})=g(X_{\tau^*_V})\]
on $\{\tau^*_V<\infty\}$
since $\Pi^{1,*}_{\tau^*_V}=Z^{p_1,x}_{\tau^*_V}=0$ due to \eqref{eq:taugamma} in the proof of Lemma \ref{lem:gamma}.

Then, putting $\tau=\tau^*_u$ in \eqref{eq:NJ1}, the inequality in \eqref{eq:NJ2} becomes an equality and we find
\begin{align}\label{eq:NJ5}
\J_1(\tau^*_u,\gamma^*;p_1,x)=\E_x[N_{\tau^*_u}]=(1-p_1)V(x),
\end{align}
where the final equality follows from \eqref{eq:NJ02}.

Integrating for $u\in[0,1)$ and recalling that $\tau^*=\inf\{t\ge 0\,:\,\Gamma^{1,*}>U_1\}$ (see Definition \ref{def:rst}) we finally obtain
\[
\J_1(\tau^*,\gamma^*;p_1,x)=(1-p_1)V(x)\ge \sup_{\tau\in\T^R_1}\J_1(\tau,\gamma^*;p_1,x).
\]
Hence, $\tau^*$ is an optimal response to $\gamma^*$.
\vspace{+4pt}

{\em Step 2.} (Optimality of $\gamma^*$.) It remains to check that $\gamma^*$ is an optimal response for Player 2 to Player 1's use of $\tau^*$. 
From \eqref{eq:J2}, for any $\zeta\in\T$ we obtain
\begin{eqnarray}\label{eq:zeta}
\J_2(\tau^*,\zeta;p_2,x)
&=&\E_x\left[e^{-r\zeta}(1-p_2\Gamma^{1,*}_\zeta)g(X_\zeta)1_{\{\zeta<\infty\}}\right]\\
&&+\frac{p_2}{2}\E_x\left[e^{-r\zeta}g(X_\zeta)\Delta\Gamma^{1,*}_\zeta1_{\{\zeta<\infty\}}\right]\notag\\
&=&\E_x\left[1_{\{\zeta<\tau^*_V\}}e^{-r\zeta}(1-p_2\Gamma^{1,*}_\zeta)g(X_\zeta)1_{\{\zeta<\infty\}}\right]\notag\\
&&+\E_x\left[1_{\{\zeta\ge \tau^*_V\}}e^{-r\zeta}(1-p_2)g(X_\zeta)1_{\{\zeta<\infty\}}\right]\notag\\
&&+\frac{p_2}{2}\E_x\left[e^{-r\zeta}g(X_\zeta)\Delta\Gamma^{1,*}_\zeta1_{\{\zeta<\infty\}}\right]\notag\\
&=&\E_x\left[1_{\{\zeta<\tau^*_V\}}e^{-r\zeta}(1-p_1\Gamma^{2,*}_\zeta)g(X_\zeta)1_{\{\zeta<\infty\}}\right]\notag\\
&&+\E_x\left[1_{\{\zeta\ge \tau^*_V\}}e^{-r\zeta}(1-p_1)g(X_\zeta)1_{\{\zeta<\infty\}}\right]\notag\\
&&+\E_x\left[1_{\{\zeta\ge \tau^*_V\}}e^{-r\zeta}(p_1-p_2)g(X_\zeta)1_{\{\zeta<\infty\}}\right]\notag\\
&&+\frac{p_2}{2}\E_x\left[e^{-r\zeta}g(X_\zeta)\Delta\Gamma^{1,*}_\zeta1_{\{\zeta<\infty\}}\right],\notag
\end{eqnarray}
where for the third equality we used  
$p_2\Gamma^{1,*}_\zeta=p_1\Gamma^{2,*}_\zeta$ on the event $\{\zeta<\tau^*_V\}$. 
Recalling that $p_1\leq p_2$, we have
\begin{eqnarray}\label{eq:zeta1}
&&\hspace{-15mm}\E_x\left[1_{\{\zeta\ge \tau^*_V\}}e^{-r\zeta}(p_1-p_2)g(X_\zeta)1_{\{\zeta<\infty\}}\right]\\
\notag
&=& -p_2\E_x\left[1_{\{\zeta\geq \tau^*_V\}}e^{-r\zeta}(1-p_1/p_2)g(X_\zeta)1_{\{\zeta<\infty\}}\right]\notag\\
&\leq & -p_2\E_x\left[1_{\{\zeta\ge \tau^*_V\}}e^{-r\zeta}\Delta\Gamma^{1,*}_\zeta g(X_\zeta)1_{\{\zeta<\infty\}}\right],\notag
\end{eqnarray}
where the inequality uses that 
\[\Delta\Gamma^{1,*}_{\zeta}1_{\{\zeta\ge \tau^*_V\}}=\Delta\Gamma^{1,*}_{\zeta}1_{\{\zeta = \tau^*_V\}}=(1-p_1/p_2)1_{\{\zeta = \tau^*_V\}}\leq (1-p_1/p_2)1_{\{\zeta \geq \tau^*_V\}}\]
on $\{\zeta<\infty\}$ by (iii) in Proposition~\ref{propDelta}.

Combining \eqref{eq:zeta1} with \eqref{eq:zeta} gives
\begin{eqnarray*}
\J_2(\tau^*,\zeta;p_2,x)
&\le &\E_x\left[e^{-r\zeta}(1-p_1\Gamma^{2,*}_\zeta)g(X_\zeta)1_{\{\zeta<\infty\}}\right]\\
&&\hspace{-10mm}+\frac{p_2}{2}\Gamma^{1,*}_0 g(x)1_{\{\zeta=0\}}
-\frac{p_2}{2}\E_x\left[e^{-r\zeta}g(X_\zeta)\Delta\Gamma^{1,*}_\zeta 1_{\{\zeta=\tau^*_V<\infty\}}\right]\notag
\end{eqnarray*}
upon recalling that $\Gamma^{1,*}$ only jumps at time zero and at $\tau^*_V$. Next, recalling (ii) from Proposition \ref{propDelta} and using \eqref{eq:J1} and $p_2\Gamma^{1,*}_0=p_1 \Gamma^{*}_0$ we obtain
\begin{eqnarray*}
\J_2(\tau^*,\zeta;p_2,x) &\le& \J_1(\zeta,\gamma^*;p_1,x)-\frac{p_2}{2}\E_x\left[e^{-r\zeta}g(X_\zeta)\Delta\Gamma^{1,*}_\zeta 1_{\{\zeta=\tau^*_V<\infty\}}\right]\\
&\le& \J_1(\zeta,\gamma^*;p_1,x).
\end{eqnarray*}
It then follows, as in \eqref{eq:NJ3} (recall also \eqref{eq:ran2}), that
\begin{align}\label{eq:NJ4}
\sup_{\zeta\in\T^R_2}\J_2(\tau^*,\zeta;p_2,x)\le (1-p_1)V(x).
\end{align}

In order to prove that equality holds  we choose $\zeta=\gamma^*_u$, with $\gamma^*_u$ as in \eqref{eq:tau-u} but with $\Gamma^{2,*}$ in place of $\Gamma$. We first notice that
\begin{align}\label{eq:NJ6}
&\J_2(\tau^*,\gamma^*_u;p_2,x)\\
&=\E_x\left[e^{-r\gamma^*_u}(1-p_1\Gamma^{2,*}_{\gamma^*_u})g(X_{\gamma^*_u})1_{\{\gamma^*_u<\infty\}}\right]+\frac{p_1}{2}g(x)\Gamma^{2,*}_01_{\{\gamma^*_u=0\}}\notag\\
&= \J_1(\gamma^*_u,\gamma^*;p_1,x)\notag
\end{align}
where we used that $p_2\Gamma^{1,*}_{\gamma^*_u}=p_1\Gamma^{2,*}_{\gamma^*_u}$, since $\gamma^*_u<\tau^*_V$
on $\{\tau^*_V<\infty\}$, $\P_x$-a.s. for $u\in[0,1)$.
Next, it is clear that \eqref{eq:equal} holds with $\gamma^*_u$ in place of $\tau^*_u$ because 
$\gamma^*_u=\tau^*_{\frac{p_1}{p_2}u}$.
Finally, letting $\tau=\gamma^*_u$ in \eqref{eq:NJ1}, the same argument that led to \eqref{eq:NJ5}, combined with \eqref{eq:NJ6}, gives
\[\J_2(\tau^*,\gamma^*_u;p_2,x)= (1-p_1)V(x).\]

Integrating over $u\in[0,1)$ and recalling that $\gamma^*=\inf\{t\ge0\,:\,\Gamma^{2,*}_t>U_2\}$ we finally conclude that
\[\sup_{\gamma\in\T^R_2}\J_2(\tau^*,\gamma;p_2,x)=\J_2(\tau^*,\gamma^*;p_2,x)= (1-p_1)V(x).\]
Hence, $\gamma^*$ is an optimal response to $\tau^*$, which completes the proof.
\end{proof}

\subsection{Equilibrium for $(p_1,x)\in\S$}
It only remains to consider an equilibrium for the game starting with $(p_1,x)\in\S$.

\begin{theorem}\label{thm:NE3}
Let $(p_1,x)\in\S$ be given and fixed. 
Then the strategy pair $(\tau^*,\gamma^*)=(0,0)$, $\P_x$-a.s.~is a Nash equilibrium. Moreover, the equilibrium payoff of the $i$-th Player is $(1-p_i/2)g(x)$, for $i=1,2$.
\end{theorem}

\begin{proof}
Fix $(p_1,x)\in\S$ and assume $\gamma^*=0$. Then for any $\tau\in\T$ we have
\begin{eqnarray}\label{eq:S0}
\J_1(\tau,0;p_1,x)
&=&(1-p_1)\E_x\left[e^{-r\tau}g(X_\tau)1_{\{\tau<\infty\}}\right]+\frac{p_1}{2}g(x)1_{\{\tau=0\}}\\
&=&\notag(1-p_1)1_{\{\tau>0\}}\E_x\left[e^{-r\tau}g(X_\tau)1_{\{\tau<\infty\}}\right]\\
&& +\left(1-\frac{p_1}{2}\right)g(x)1_{\{\tau=0\}}\notag\\
&\le& (1-p_1)1_{\{\tau>0\}}V(x)+\left(1-\frac{p_1}{2}\right)g(x)1_{\{\tau=0\}}\notag\\
&\le&\left(1-\frac{p_1}{2}\right)g(x),\notag
\end{eqnarray}
where the final inequality follows from $(1-p_1)V(x)\le (1-p_1/2)g(x)$. 

Thanks to \eqref{eq:S0} and \eqref{eq:ran1} we have
\[
\sup_{\tau\in\T^R_1}\J_1(\tau,0;p_1,x)\le \left(1-\frac{p_1}{2}\right)g(x).
\]
It is now straightforward to verify that equality holds if we choose $\tau=0$, $\P_x$-a.s., so $\tau^*=0$ is an optimal response to $\gamma^*=0$.

A similar argument allows to prove that $\gamma^*=0$ is also an optimal response to $\tau^*=0$. 
Indeed, as in \eqref{eq:S0} we have for $\gamma\in\T$ that
\begin{eqnarray*}\label{eq:S1}
\J_2(0, \gamma;p_2,x)
&=&\notag(1-p_2)1_{\{\gamma>0\}}\E_x\left[e^{-r\gamma}g(X_\gamma)1_{\{\gamma<\infty\}}\right]\\
&& +\left(1-\frac{p_2}{2}\right)g(x)1_{\{\gamma=0\}}\notag\\
&\le& (1-p_2)1_{\{\gamma>0\}}V(x)+\left(1-\frac{p_2}{2}\right)g(x)1_{\{\gamma=0\}}\notag\\
&\le&\left(1-\frac{p_2}{2}\right)g(x)\notag
\end{eqnarray*}
since $(1-p_1)V(x)\le (1-p_1/2)g(x)$ and $p_2\geq p_1$ imply that $(1-p_2)V(x)\le (1-p_2/2)g(x)$.
Moreover, equality holds for $\gamma=0$, and thus
$(\tau^*,\gamma^*)=(0,0)$ is a Nash equilibrium as claimed.
\end{proof}

\begin{remark}
Notice that the equilibrium value for Player~2 is not continuous in $p_1$ at $c(x)$. In fact, for 
$p_1<c(x)<1$, the equilibrium value for Player~2 is $(1-p_1)V(x)$ (see Theorem~\ref{thm:NE}), whereas 
the equilibrium value for $p_1=c(x)$ in Theorem~\ref{thm:NE3} is $(1-p_2/2)g(x)\leq (1-p_1/2)g(x)=(1-p_1)V(x)$, where 
the inequality is strict if $p_2>p_1$.
\end{remark}


\subsection{Comments on our results and heuristics}\label{sec:comment}
\subsubsection*{Benefits of wariness} 
In the case of known competition, i.e.~$p_1=p_2=1$, both players stopping immediately provides a 
Nash equilibrium, and in the case of no competition, one should use the single-player stopping strategy $\tau^*_V$.
In view of this, it is intuitively clear that Player~2, who is more `wary' (or aware) of competition than Player~1,
should be the player who is more eager to stop early. This intuition is confirmed by Theorem~\ref{thm:NE}; 
in fact, the more active role of Player 2 is quantified by the fact that she stops with 
`generalised intensity' which is $(p_2/p_1)$-times bigger than that of Player 1. 

Player 2's equilibrium strategy is surprisingly rewarding because it yields an average payoff $(1-p_1)V(x)$, which is strictly larger than her `safety' level $(1-p_2)V(x)$, if $p_1<p_2$ (see \eqref{U}). In contrast, the equilibrium strategy of Player 1 yields just the `safety' level $(1-p_1)V(x)$. 
Interestingly, this implies that Player~1 is not worse off by publicly revealing her presence in the game. 
In fact, consider two games with parameters $(p_1,p_2)$ and $(p_1,1)$, respectively (and assume that $p_1<c(x)$ so that we are in the setting of Theorem~\ref{thm:NE}).
Then both Players' equilibrium payoffs would, in both cases, be $(1-p_1)V(x)$.

\subsubsection*{Jumps in equilibrium strategies} A truly remarkable feature from the mathematical point of view is the size of the initial jump in the equilibrium process $\Gamma^{2,*}$ for an initial belief $p_1\in (b(x),c(x))$. Indeed, we have seen in Theorem~\ref{thm:NE} that such jump pushes $(\Pi^{1,*},X)$ strictly in the interior of the set $\OC$ (see Lemma \ref{lem:5.1}). This is somewhat unexpected if we consider that our set-up shares similarities with the framework of nonzero-sum games of singular control with proportional cost of control (see, e.g., \cite{DeAFe18}). In those problems the initial jump of the optimal control drives the optimally controlled process to the boundary of the action set. Hence, in our setting we might have expected a jump to the boundary of the set $\OC$. 

We note that the observed structure of our optimal controls at equilibrium might stem from the fact that players split the payoff if stopping simultaneously. 

\subsubsection*{Heuristics Theorem \ref{thm:NE}}
For simplicity (and following the actual route that led us to our conclusions) we consider the case of players with {\em symmetric} initial beliefs, i.e.~$p_1=p_2=:p$. In this case the players should use the same strategy 
$\Gamma^{1,*}=\Gamma^{2,*}=\Gamma^{*}$.
Moreover, if Player~2 uses a randomised time $\gamma^*$ generated by $\Gamma^*$, then, for optimality, Player~1 should be indifferent between strategies 
$\tau^*_u=\inf\{t\ge 0:\Gamma^*_t>u\}$ as in \eqref{eq:tau-u} with $u\in[0,1)$ (see Proposition \ref{altform}). In particular, one would expect that $\J_1(\tau_u^*,\gamma^*;p,x)$ is independent of $u\in[0,1)$, so the equilibrium value should be given by
\[\J_1(\tau_u^*,\gamma^*;p,x)= \lim_{u\uparrow 1}\J_1(\tau_u^*,\gamma^*;p,x) =\J_1(\tau_V^*,\gamma^*;p,x)= (1-p)V(x)\]
(where the second and third equalities are to be understood intuitively). 
With this candidate equilibrium value in mind, the construction in Theorem~\ref{thm:NE} follows naturally.

We next comment on the specification of $\Gamma^*_0$ in the case $b(x)< p <c(x)$.
Again we consider players with symmetric initial beliefs, i.e.~$p_1=p_2=:p$ so that, at equilibrium, they should use the same strategy. 

Assume Player 2 chooses to stop immediately with probability $\Gamma^2_0$. For this to be an equilibrium strategy,  Player 1 must be indifferent between stopping and continuing (so that any linear combination of such strategies gives a best reply). The indifference condition reads 
\begin{align}\label{indiff}
&(1-p)g(x)+p(1-\Gamma^2_0)g(x)+\frac{p}{2}\Gamma^2_0g(x)=(1-p+p(1-\Gamma^2_0))(1-q)V(x)
\end{align}
where, on the left-hand side of the equation we have Player 1's payoff in case of immediate stopping and, on the right-hand side, her payoff if continuing. Notice in particular that 
\[
q=\frac{p(1-\Gamma^2_0)}{1-p\Gamma^2_0}
\]
(defined as in \eqref{eq:Gq}) is the adjusted belief of Player 1 if Player 2 does not end the game immediately (see \eqref{Pi1}). On the right-hand side of \eqref{indiff} we have $(1-q)V(x)$ because we expect that, if the game does not end immediately, the belief should be updated to $(q,x)\in\OC$.

Solving the indifference equation for $\Gamma^2_0$ gives us the first equation in \eqref{eq:Gq}.

\subsubsection*{Evolution of adjusted beliefs}
Defining $\Pi^{2,*}_t:= \frac{p_2(1-\Gamma^{1,*}_t)}{1-p_2\Gamma^{1,*}_t}$, it is straightforward to check that 
\[\Pi^{2,*}_t= \frac{(1-p_2)\Pi^{1,*}_t + p_2-p_1}{1-p_1}.\]
Consequently, 
\[0\leq \Pi^{2,*}_t-\Pi_t^{1,*}=\frac{p_2-p_1}{1-p_1}(1-\Pi_t^{1,*}).\]
Thus, in equilibrium, both adjusted beliefs are nonincreasing, but the difference between them is (rather surprisingly) nondecreasing.

\subsubsection*{Connection with barrier strategies}
As is shown in Theorem~\ref{thm:NE}, if $(p_1,x)\in\overline\C$, the process $X$ is stopped (in equilibrium) only when $t\mapsto b(X_t)$ reaches a new minimum. That is, $X$ is stopped (with some generalised intensity) upon reaching for the first time a new region in $\R^d$. This is a remarkable feature which shows that if the game 
was specified so that the set of strategies consisted only of entry times to (randomised) sets then the same equilibrium would be obtained. The conclusion is more easily visualised in a one dimensional state space (as in the example in the next section). If $X\in\R$ we would have Nash equilibria produced by threshold strategies (possibly two-sided), with random thresholds.
Note, however, that our methodology shows that $(\tau^*,\gamma^*)$ in fact is a Nash equilibrium in the much larger class of strategies that consists of all randomised times. This fact is far from being trivial because, if we think of equilibria in the game as a fixed point for each player's best response to the opponent's strategy, there is no reason to expect that equilibrium threshold strategies should emerge from the much broader class that we consider.

\section{Real options with unknown competition}\label{sec:realoption}

In this section we consider an example of a real option with uncertain competition (for the classical case of real options 
with no competition, see e.g. \cite{DP}, and for a related problem of real option pricing under incomplete information about the competition, see \cite{LP}).
For this, let $g(x)=(x-K)^+$, $p_1=p_2=:p\in(0,1)$ and
\[\ud X_t=\mu X_t\,\ud t + \sigma X_t\,\ud W_t\]
for some constants $K>0$, $\mu<r$ and $\sigma>0$.
Here $X$ represents the present value of future revenues from 
entering a certain business opportunity, $K$ is the sunk cost for entering, and $r$ is a discount rate.
The business opportunity, however, is subject to competition,
and we assume that each active agent estimates the probability of an active competitor 
to be $p$ (hence the probability of no competition is $1-p$). This symmetric setting 
is reasonable in applications where players are `similar' (e.g., firms of a similar size producing the same good).

It is well-known (e.g. \cite{DP}) that the value 
\[V(x):=\sup_{\tau}\E_x[e^{-r\tau}g(X_\tau)1_{\{\tau<\infty\}}]\]
in the corresponding one-player game is given by 
\[
V(x)=
\left\{
\begin{array}{ll}
(B-K)(x/B)^\eta, & \text{for $x\in(0,B)$,}\\
g(x), & \text{for $x\in[B,\infty)$},
\end{array}\right.\]
where 
\[\eta=\frac{\sigma^2-2\mu}{2\sigma^2} + \sqrt{\left(\frac{\sigma^2-2\mu}{2\sigma^2}\right)^2+\frac{2r}{\sigma^2}}\in(1,\infty)\]
and $B:=\eta K/(\eta-1)$.
Since $V$ and $B$ are explicit, recalling \eqref{eq:b(x)} and \eqref{eq:c(x)} one can obtain explicit formulae for the boundaries $x\mapsto b(x)$ and $x\mapsto c(x)$ of the sets $\overline \C$ and $\C^\prime$ in \eqref{eq:C2} and \eqref{eq:C3}.
In fact, 
\[b(x)=1-\frac{g(x)}{V(x)}=\left\{\begin{array}{cl}
1 & x\in(0,K]\\
1-\frac{(x-K)(B/x)^\eta}{B-K} & x\in(K,B)\\
0 & x\in[B,\infty)
\end{array}\right.\]
and
\[c(x)=\frac{V(x)-g(x)}{V(x)-g(x)/2}=\left\{\begin{array}{cl}
1 & x\in(0,K]\\
1-\frac{(x-K)}{2(B-K)(x/B)^\eta-x+K} & x\in(K,B)\\
0 & x\in [B,\infty).\end{array}\right.\]
For a graphical illustration of these functions, see Figure~\ref{fig:b}. 

In line with Lemma~\ref{cont} we note that
\[
\lim_{x\to B}b(x)=\lim_{x\to B}c(x)=0 \quad\text{and}\quad\lim_{x\to K}b(x)=\lim_{x\to K}c(x)=1.
\]
Moreover, it is straightforward to check that $b$ and $c$ are non-increasing, and they are strictly decreasing on $(K,B)$.
Hence, for any $(p,x)\in\overline\C$ we also have
\begin{align}\label{eq:exZ}
Z^{p,x}_t=\inf_{0\le s \le t}b(X^x_s)=b\left(\sup_{0\le s\le t}X^x_s\right),\qquad t\ge0,
\end{align} 
by \eqref{eq:SIZ}. Thanks to these formulae we can now write the Nash equilibrium in a very explicit form. We give some details when the game starts in $\overline\C$ and the other cases can be handled analogously.

\begin{figure}[htp!]\centering
  \includegraphics[width=0.45\textwidth]{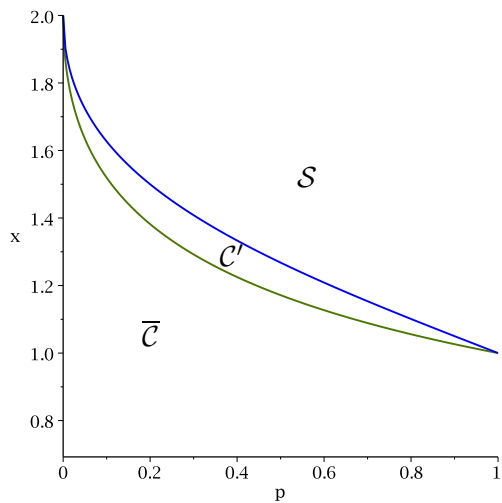}
\put(-109, 91){$\mathcal C'$}
\put(-80, 121){$\mathcal S$}
\put(-138, 61){$\OC$}
  \caption{The figure displays the curves $p=b(x)$ (lower one) and $p=c(x)$ (top one). The parameter values are $K=1$ and $\eta=2$ so that $B=2$.} \label{fig:b}
\end{figure}

Given an initial point $(p,x)\in \overline\C$ (i.e. $p\leq b(x)$), Theorem~\ref{thm:NE} applies with $\Gamma^*_0=0$. 
Moreover, we notice that $\Gamma^{1,*}=\Gamma^{2,*}=:\Gamma^*$, due to symmetric beliefs of the players. Thanks to \eqref{eq:exZ} and \eqref{eq:Gpx}, and recalling that $\tau^*$ and $\gamma^*$ are generated by $\Gamma^*$, we have
\begin{align*}
\tau^*=&\inf\{t\ge 0: \Gamma^*_t>U_1\}\\
=&\inf\{t\ge 0:Z_t< p(1-U_1)/(1-pU_1)\}\\
=&\inf\{t\ge 0: X_t\ge B^*_1\},
\end{align*}
and $\gamma^*=\inf\{t\ge 0: X_t\ge B^*_2\}$, with the random variables 
\begin{equation}
\label{B*}
B^*_i:=b^{-1}\left(\frac{p(1-U_i)}{1-p U_i}\right), \quad i=1,2,
\end{equation}
where $b^{-1}$ is the inverse function of $b$.

\end{document}